\documentclass[runningheads]{llncs}
\usepackage{url}
\usepackage{multirow}
\usepackage{amssymb}
\usepackage{amsxtra}
\usepackage{makeidx}
\usepackage{amsfonts}
\usepackage{mathtools}
\usepackage{algorithm}
\usepackage{algorithmic}
\usepackage{float}
\usepackage{graphicx}
\usepackage{comment}
\usepackage[font=small,labelfont=bf]{caption}
\usepackage{subcaption}
\usepackage{breqn}
\usepackage[english]{babel}

\usepackage{color}
\usepackage{hyperref}

\DeclareUnicodeCharacter{202A}{\,}

\begin{document}
	
\title{Algorithm for solving variational inequalities with relatively strongly monotone operators}

\author{Alexander A. Titov}
\institute{Moscow Institute of Physics and Technology, Moscow, Russia}

	


\maketitle
\begin{abstract}
Basing on some recently proposed methods for solving variational inequalities with non-smooth operators, we propose an analogue of the Mirror Prox method for the corresponding class of problems under the assumption of relative smoothness and relative strong monotonicity of the operator.


\end{abstract}

\section*{Introduction}\label{sec1_introduction}
Recently \cite{Main} there were proposed some
numerical methods for solving saddle point problems and variational inequalities with simplified requirements for the smoothness conditions of functionals. By simplified smoothness conditions we mean, generally speaking, non-smooth operators, which  do not satisfy the Lipschitz condition, but satisfy some of its weakened versions. We base on the novel work \cite{Inex}, in which the authors successfully transferred the ideology of gradient methods to the case of relatively strongly convex and relatively smooth objective.
We continue developing similar ideas and consider variational inequalities with relatively smooth and relatively strongly monotone operators. Such operators naturally arise when considering the relatively strongly convex saddle point problem and reducing it to the variational inequality \cite{Nesterov}.

Let $E$ be some finite-dimensional vector space and $E^*$ be its dual. For a fixed  norm $\|.\|$ on $E$ define the corresponding dual norm $\|\cdot\|_*$ as follows:
$$\|\phi\|_{*} = \max\limits_{\|x\|\leq 1} \{ \langle \phi, x \rangle\},$$
where $\langle \phi, x \rangle$ denotes the value of the linear function $\phi \in E^*$ at the point $x\in E$.

Let $X\subset E$ be a compact set.

Consider relatively smooth, relatively strongly convex, and relatively strongly monotone operator $g(x):X\rightarrow E^*$, satisfying

\begin{enumerate}
    \item Inexactness \begin{equation}\label{eq1}
    \langle g(y),x-y \rangle\leq \langle g_{\delta}(y),x-y \rangle +\delta 
    \end{equation}
    \item Relative strong monotonicity
    \begin{equation}\label{eq2}
    \langle g_{\delta}(y),x-y\rangle + \langle g_{\delta}(x),y-x\rangle +\mu V(x,y) \leq \delta 
    \end{equation}
    \item Relative smoothness
    \begin{equation}\label{eq3}
    \langle g_{\delta}(y)-g_{\delta}(z),x-z\rangle \leq LV(x,z) + LV(z,y) +\delta
    \end{equation}
\end{enumerate}
\begin{definition}[Minty Variational Inequality]
     For a given operator\\ $g(x):X\rightarrow\mathbb{R}$ we need to find a vector $x_*\in X$, such that
\begin{equation}\label{VI}
  \langle g(x),x_*-x\rangle\leq 0, \quad \forall x\in X.
\end{equation}
    \end{definition}
We also need to choose a prox-function $d(x)$, which is continuously differentiable and 1-strongly convex on $X$, and the corresponding Bregman distance, defined as follows: 
\begin{equation}\label{Bregmann_def}
    V(y,x) = d(y) - d(x) - \langle \nabla d(x), y - x \rangle, \quad \forall x, y \in X.
\end{equation}

Consider the following Mirror Prox algorithm \cite{Inex}.

\begin{algorithm}[h!]
\caption{Universal Mirror Prox (\tt{UMP})}
\label{Alg:UMP}
\begin{algorithmic}[1]
   \REQUIRE $\varepsilon > 0$, $\delta >0$, $x_0 \in X$, initial guess $L_0 >0$, prox-setup: $d(x)$, $V(x,z)$.
   \STATE Set $k=0$, $z_0 = \arg \min_{u \in Q} d(u)$.
   \REPEAT
			\STATE Find the smaller $i_k\geq 0:$ 
			\begin{equation}
			\begin{multlined}
			    \langle g_{\delta}(z_k),z_{k+1}-z_k\rangle \leq \langle g_{\delta}(w_k),z_{k+1}-w_k\rangle + \langle g_{\delta}(z_k),w_{k}-z_k\rangle \\+L_{k+1}(V(w_k,z_k) + V(z_{k+1},w_k)) +\delta
			    \end{multlined}
			\end{equation}
			
			\STATE where $L_{k+1} = 2^{i_k-1}L_k$ and 
    			\begin{equation}
    			     w_k = \arg\min\limits_{x}\{\langle g_{\delta}(z_k),x-z_k\rangle + L_{k+1}V(x,z_k)\}
    			\end{equation}
    				\begin{equation}
    			     z_{k+1} = \arg\min\limits_{x}\{\langle g_{\delta}(w_k),x-w_k\rangle + L_{k+1}V(x,z_k)\}
    			\end{equation}
			\UNTIL{
			\begin{equation}
		        S_N = \sum\limits_{k=0}^{N-1}\frac{1}{L_{k+1}}\geq \frac{\max\limits_{x\in X}V(x,x_0)}{\varepsilon}
			\end{equation}}
		\ENSURE $z^k$
\end{algorithmic}
\end{algorithm}
\begin{theorem}
It is well-known \cite{Inex}, that for the output of the Mirror Prox algorithm the following inequality takes place:
$$
\langle g(x_*),z_k-x_*\rangle \leq -\frac{1}{S_N}\sum\limits_{k=0}^{N-1}\frac{\langle g(w_k),x_*-w_k \rangle}{L_{k+1}} \leq \frac{2L V(x_*,z_0)}{N},
$$
moreover, the total number of iterations does not exceed $$N = \frac{2L\max\limits_{x\in X}V(x_0,x)}{\varepsilon}.$$
\end{theorem}

\begin{lemma}
For the  described operator $g$ and Mirror Prox algorithm there takes place the following $\delta-$decreasing of Bregman divergence:
\begin{equation}
V(x_*,z_{N})\leq V(x_*,z_{0}) +\delta S_N.
\end{equation}
\end{lemma}
\begin{proof}
It is known (\cite{Inex}, proof of Theorem 4.8), that $\forall k\geq 0$
\begin{equation}
    -\langle g_{\delta}(w_k),u-w_k \rangle\leq L_{k+1}V(u,z_k)-L_{k+1}V(u,z_{k+1}) +\delta
\end{equation}
Due to the relatively strong monotonicity of the operator $g$, we can consider, that the solution to weak variational inequality is also a strong solution:
\begin{equation}
-\langle g(x_*),w_k-x_*\rangle\leq 0,
\end{equation}
which, due to \eqref{eq1}, leads to 
\begin{equation}
-\langle g_{\delta}(x_*),w_k-x_*\rangle\leq \delta,
\end{equation}
and
			\begin{equation}
			\begin{multlined}
			    0\leq -\langle g_{\delta}(x_*),x_*-w_k\rangle +\delta \leq-\langle g_{\delta}(w_k),x_*-w_k \rangle\leq\\ \leq L_{k+1}V(x_*,z_k)-L_{k+1}V(x_*,z_{k+1}) +\delta.
			    \end{multlined}
			\end{equation}

So, there takes place the following $\delta-$decreasing of the Bregman divergence:
\begin{equation}
V(x_*,z_{k+1})\leq V(x_*,z_{k}) +\frac{\delta}{L_{k+1}} \quad \forall k.
\end{equation}

Thus,
\begin{equation}
V(x_*,z_{N})\leq V(x_*,z_{0}) +\delta\sum\limits_{i=1}^N\frac{1}{L_i}.
\end{equation}

\end{proof}

Let us now consider the following algorithm, which is in fact the restarted version of the considered Mirror Prox algorithm.

\begin{algorithm}[htp]
	\caption{Restarted Universal Mirror Prox ({\tt Restarted UMP}).}
	\label{Alg:RUMP}
	\begin{algorithmic}[1]
		\REQUIRE $\varepsilon > 0$, $\mu >0$, $\Omega$ : $d(x) \leq \frac{\Omega}{2} \ \forall x\in Q: \|x\| \leq 1$; $x_0,\; R_0 \ : \|x_0-x_*\|^2 \leq R_0^2.$
		\STATE Set $p=0,d_0(x)=R_0^2d\left(\frac{x-x_0}{R_0}\right)$.
		\REPEAT
		\STATE Set $x_{p+1}$ as the output of {\tt UMP} for monotone case with prox-function $d_{p}(\cdot)$ and stopping criterion $\sum_{i=0}^{k-1}M_i^{-1}\geq \frac{\Omega}{\mu}$.
		\STATE Set $R_{p+1}^2 = \frac{\Omega R_0^2}{2^{(p+1)}\mu S_{N_p}} -\frac{\delta}{\mu}$.
		\STATE Set $d_{p+1}(x) \leftarrow R_{p+1}^2d\left(\frac{x-x_{p+1}}{R_{p+1}}\right)$.
		\STATE Set $p=p+1$.
		\UNTIL			
		$p > \log_2\left(\frac{2R_0^2}{\varepsilon}\right).$	
		\ENSURE $x_p$.
	\end{algorithmic}
\end{algorithm}

\begin{theorem}
Consider relatively smooth and relatively strongly monotone operator $g$, satisfying \eqref{eq1}-\eqref{eq3}. Then the restarted version of Mirror Prox algorithm produces the point $x_p$, such that
$$V(x_*,x_p)\leq \varepsilon+\frac{\delta}{\mu}\left(1+\frac{2 \Omega L}{\mu}\right),$$
moreover, the total number of iterations does not exceed $$N = \frac{2L\Omega}{\mu}\log_2\frac{R_0^2}{\varepsilon}.$$
\end{theorem}
\begin{proof}
Due to relatively strong monotonicity of the operator:
\begin{equation}
    \mu V(x_*,z_{p+1})\leq \delta - \langle g_{\delta}(x_*),z_{p+1}-x_* \rangle - \langle g_{\delta}(z_{p+1}),x_*-z_{p+1}\rangle \leq \\\delta - \langle g_{\delta}(z_{p+1}),x_*-z_{p+1} \rangle 
\end{equation}

Let us use induction to show the fulfillment of the theorem. Consider $p=0$. After no more than $N_0 = \frac{2L\Omega}{\mu}$ iterations of Mirror Prox algorithm we get:

\begin{equation}
     -\langle g(x_*), \underbrace{w_{N_0-1}}_{x_1}-x_*\rangle \leq 2\delta -\frac{1}{S_{N_0}}\sum\limits_{k=0}^{N_0-1}\frac{\langle g_{\delta}(w_k),x_*-w_k\rangle}{L_{k+1}}\leq \frac{\Omega R_0^2}{2S_{N_0}}-\delta := R_1^2
\end{equation}

Let us assume, that this inequality takes place for some $p>0$:

\begin{equation}
    -2\delta +\frac{\mu}{S_{N_p}}\sum\limits_{k=0}^{N_p-1}\frac{V(x_*,w_k)}{L_{k+1}}\leq -2\delta -\frac{1}{S_{N_p}}\sum\limits_{k=0}^{N_p-1}\frac{\langle g_{\delta}(w_k),x_*-w_k\rangle}{L_{k+1}}\leq \frac{\Omega R_p^2}{2S_{N_p}}-\delta
\end{equation}
and prove, that the same holds for $p+1$.
$$
\sum_{k=0}^{N_{p}-1} \frac{V\left(x_{*}, w_{k}\right)}{L_{k+1}} \geqslant V\left(x_{*}, w_{N_p-1}\right)\cdot S_{N_p} -\frac{\delta}{L_{N_{p}}}-\left(\frac{\delta}{L_{N_p}}+\frac{\delta}{L_{N_p-1}}\right)-\ldots-\delta S_{N_{p}}=
$$
$$
=S_{N_{p}} V\left(x_{*}, w_{N_p-1}\right)-\frac{\delta N_p}{L_{N_p}} - \frac{\delta (N_p-1)}{L_{N_p-1}} - \frac{2\delta}{L_2} - \frac{\delta}{L_1}
$$

Thus,
$$-2 \delta+\frac{\mu}{S_{N_p}} \sum_{k=0}^{N_{p}-1} \frac{V\left(x_{*}, w_{k}\right)}{L_{k+1}} \geqslant-2 \delta + \mu V\left(x_{*}, w_{N_{p-1}}\right)-$$
$$
\frac{\delta}{S_{N_{p}}}
\biggl(\frac{N_{p}}{L_{N_p}}+\frac{N_{p}-1}{L_{N_{p}-1}}+\ldots +\frac{2}{L_2}+\frac{1}{L_1}\biggl)\geq \mu V(x_*,w_{N_p-1}) - (2+N_p)\delta
$$
So,
$$
V\left(x_{*}, x_{p+1}\right) \leqslant \frac{\Omega R_{p}^{2}}{ 2\mu S_{N_{p}}}+\frac{\delta+\delta N_p}{\mu}
\leqslant  \frac{\Omega R_{p}^{2}}{ 2\mu S_{N_{p}}}+\frac{\delta}{\mu}\left(1+\frac{2 \Omega L}{\mu}\right),
$$
which ends the proof of the theorem.

\end{proof}

\end{document}